\documentclass[letterpaper, 10pt, conference]{ieeeconf}

\input{TeX/Preamble.tex}

\newcommand{\TheFunding}{This work was supported by the Japan Society for the Promotion of Science (JSPS) KAKENHI grants JP21K17710 and JP24K20737}
\newcommand{\TheTitle}{Safeguarding adaptive methods: global convergence of Barzilai-Borwein and other stepsize choices}

\title{\LARGE\bf\TheTitle}

\author{Hongjia Ou and Andreas Themelis%
	\thanks{%
		\TheAddressKUJ.\newline
		{\em Email:}
		\scriptsize
		\sf
		\emailLink{ou.honjia.069@s.kyushu-u.ac.jp},
		\emailLink{andreas.themelis@ees.kyushu-u.ac.jp}%
	}%
	\thanks{\TheFunding.}%
}

	\showgrayfalse

\renewcommand{\includetikz}[2][]{\includegraphics[#1]{Pics/Tikz/#2.pdf}}

\begin{document}

	\maketitle

	\begin{abstract}
		Leveraging on recent advancements on adaptive methods for convex minimization problems, this paper provides a linesearch-free proximal gradient framework for globalizing the convergence of popular stepsize choices such as Barzilai-Borwein and one-dimensional Anderson acceleration.
This framework can cope with problems in which the gradient of the differentiable function is merely locally H\"older continuous.
Our analysis not only encompasses but also refines existing results upon which it builds.
The theory is corroborated by numerical evidence that showcases the synergetic interplay between fast stepsize selections and adaptive methods.

	\end{abstract}

	\section{Introduction}
		Convex nonsmooth optimization problems are encountered in various engineering applications such as image denoising \cite{beck2009fast}, signal processing and digital communication \cite{luo2003applications}, machine learning \cite{bubeck2014theory}, and control \cite{li2022role}, to name a few.
Many such problems can be cast in composite form as
\[\tag{P}\label{eq:P}
	\minimize_{\x\in\R^n}\varphi(\x)\coloneqq f(\x)+g(\x),
\]
where \(\func{f}{\R^n}{\R}\) is convex and differentiable, and \(\func{g}{\R^n}{\R\cup\set{\infty}}\) is proper and lower semicontinuous (lsc).
A textbook algorithm in this setting is the proximal gradient method that involves updates of the form
\begin{equation}\label{eq:PG}
	\xk*
=
	\prox_{\gamk*g}(\xk-\gamk*\nabla f(\xk)),
\end{equation}
where
\begin{equation}\label{eq:prox}
	\prox_{g}(x)
\coloneqq
	\argmin_{w\in\R^n}\set{g(w)+\tfrac{1}{2}\norm{w-\x}^2}
\end{equation}
is the \emph{proximal mapping} of \(g\), available in closed form for many practical applications \cite{beck2017first}.

In \eqref{eq:PG}, the choice of the \emph{stepsize} parameters \(\gamk*\) plays a pivotal role in dictating the efficiency of the algorithm.
Traditional constant stepsizes require the gradient of the function \(f\) to satisfy global Lipschitz continuity \cite[Prop. 1.2.3]{bertsekas2016nonlinear}, and stepsizes that are either too large or too small can lead to slow convergence or failure thereof.
Alternatively, the utilization of backtracking linesearch for stepsize determination \cite[\S9.3]{boyd2004convex} provides an online tuning that works under less restrictive assumptions, but incurs computational overhead due to the inner subroutines for assessing each stepsize.

A revolutionary concept was conceived in \cite{malitsky2020adaptive}, which proposes an adaptive method for automatically tuning the stepsizes without backtracking routines to address scenarios of convex smooth optimization, that is, when \(g = 0\) in \eqref{eq:PG}, yet under mere \emph{local} (as opposed to \emph{global}) Lipschitz gradient continuity.
This pioneering work has been attracting considerable attention in the last couple of years, leading to proximal (and primal-dual) variants \cite{latafat2023adaptive,malitsky2023adaptive,latafat2023convergence}.

The present paper continues this trend, inspired by two recent interesting developments: \cite{oikonomidis2024adaptive} which reduces the assumptions on \(\nabla f\) to \emph{local H\"older continuity}, and \cite{zhou2024adabb} which incorporates Barzilai-Borwein stepsizes \cite{barzilai1988two} to enhance convergence speed.
The Barzilai-Borwein can be classified as a one-dimensional quasi-Newton method, and its convergence without backtracking routines has only been established for smooth and strongly convex quadratic problems \cite{raydan1993barzilai,dai2002rlinear,li2021faster}.
A first attempt beyond the quadratic case was advanced in \cite{burdakov2019stabilized} with the {\algnamefont stabilized BB} method, which provides a suitable dampening of the BB stepsizes ensuring convergence for strongly convex and smooth problems.
However, convergence is established only for (undefined) small enough choices of a parameter, precluding an off-the-shelf use in practice.

The above-mentioned \cite{zhou2024adabb} obviates this problem by designing a tailored algorithm, \adaPBB, that carefully integrates the adaptive ideas of \cite{malitsky2020adaptive} and successive works.
The method is an alternative way of ``stabilizing'' BB stepsizes, promoting larger stepsizes (on average) which are associated with faster convergence (see the next section).


The present paper patterns the same rationale and proposes a simpler ``stabilizing'' approach, not necessarily tied to the choice of BB stepsizes, that builds upon the recent findings of \cite{oikonomidis2024adaptive} to apply to a much broader class of problems.
Our working assumptions on problem \eqref{eq:P} are the following:

\begin{assumption}\label{ass:basic}
	The following hold in problem \eqref{eq:P}:
	\begin{enumeratass}[widest=3]
	\item \label{ass:f}%
		\(\func{f}{\R^n}{\R}\) is convex and has locally H\"older-continuous gradient of order \(\q\in(0,1]\) (see \cref{sec:recipe}).
	\item \label{ass:g}%
		\(\func{g}{\R^n}{\Rinf}\) is proper, lsc, and convex.
	\item
		A solution exists: \(\argmin\varphi\neq\emptyset\).
	\end{enumeratass}
\end{assumption}

Although \(\q>0\) is ultimately required in the analysis, in some results that are valid in the limiting case \(\q=0\) we will explicitly mention the extended range \(\q\in[0,1]\).
As formally elaborated in \cite{oikonomidis2024adaptive}, the case \(\q=0\) amounts to \(f\) being merely convex and real-valued, and the notation \(\nabla f\) is used to indicate any \emph{subgradient} selection \(\nabla f(\x)\in\partial f(\x)\).

\subsection{Contributions}
	As a means to globalize in a linesearch-free fashion a virtually arbitrary variety of stepsize choices, we employ adaptive methods as \emph{``safeguarding''} mechanisms to automatically dampen, when necessary, overshooting stepsizes that may cause divergence.
	Our focus is on \adaPG{} developed in \cite{latafat2023convergence}, as it constitutes an umbrella framework enabling us to encompass many other adaptive methods at once.

	To embrace the generality of \cref{ass:basic}, we provide a ``convergence recipe'' that identifies which are the stepsize choices that can be \emph{safeguarded}, including, but not limited to, Barzilai-Borwein \cite{barzilai1988two} and Anderson acceleration--like choices \cite{anderson1965iterative,fang2009two}.
	Even when specialized to the plain \adaPG{} algorithm, this general view results in tighter constants for the convergence rate compared to the tailored analysis of \cite{oikonomidis2024adaptive}.

	Finally, we provide a list of stepsize choices that comply with our framework, and provide numerical evidence in support of their employment.

\subsection{Paper organization}
	In \cref{sec:motivations} we provide a brief overview on adaptive methods and build on the motivating examples that led to this work; the section provides a gradual walkthrough to the idea and the rationale of \emph{``safeguarding''}.
	The technical core of the paper is in \cref{sec:recipe}, where the general convergence recipe for a class of adaptive methods is presented and validated with formal proofs.
	Some stepsize choices compliant with the recipe are listed in \cref{sec:stepsizes}, and \cref{sec:simulations} presents numerical experiments followed by some closing remarks.

	\section{Adaptive methods as safeguards}
		\label{sec:motivations}%
		In contrast to the employment of constant parameters based on global (worst-case) Lipschitz moduli, adaptive methods generate stepsizes at every iteration from \emph{local} estimates
based on \emph{past} information, thereby completely waiving any need of inner subroutines.
At iteration \(k\), the new stepsize \(\gamk*\) is retrieved based on the following quantities:
\begin{subequations}\label{eq:lL_noq}
	\begin{align}
		\ell_k
	\coloneqq{} &
		\frac{
			\innprod{\nabla f(\xk)-\nabla f(\x^{k-1})}{\xk-\x^{k-1}}
		}{
			\norm{\xk-\x^{k-1}}^2
		},
	\\
		L_k
	\coloneqq{} &
		\frac{
			\norm{\nabla f(\xk)-\nabla f(\x^{k-1})}
		}{
			\norm{\xk-\x^{k-1}}
		},
	\shortintertext{and/or}
		c_k
	\coloneqq{} &
		\frac{
			\norm{\nabla f(\xk)-\nabla f(\x^{k-1})}^2
		}{
			\innprod{\nabla f(\xk)-\nabla f(\x^{k-1})}{\xk-\x^{k-1}}
		}.
	\end{align}
\end{subequations}
When \(\nabla f\) is \emph{locally} Lipschitz, \cite{latafat2023convergence} shows that the update
\begin{equation}\label{eq:adaPG}
	\gamk*
\coloneqq
	\min\set{
		\gamk\sqrt{\tfrac{1}{\rp}+\tfrac{\gamk}{\gam_{k-1}}},
		\tfrac{\gamk}{
			\sqrt{2\left[\gamk^2L_k^2-(2-\rp)\gamk\ell_k+1-\rp\right]_+}
		}
	}
\end{equation}
for some \(\rp\in[1,2]\) guarantees, for convex problems, convergence of \eqref{eq:PG} to a solution (\([{}\cdot{}]_+\coloneqq\max\set{{}\cdot{},0}\)).
The remarkable performance in practice owes to two factors: the ability to adapt to the \emph{local} geometry of \(f\), and the consequent employment of much larger stepsizes than textbook constant ``worst-case'' choices would yield.
It is worth noting that the analysis of \cite{latafat2023convergence} and related references still holds if equality in \eqref{eq:adaPG} is replaced by a ``\(\leq\)'', as long as the stepsizes remain bounded away from zero.
This fact was used to demonstrate that the adaptive schemes \cite{malitsky2020adaptive,latafat2023adaptive,malitsky2023adaptive} proposed in previous works could be covered by the same theory, despite the update in \eqref{eq:adaPG} may be larger.
Nevertheless, as the sublinear worst-case rate
\[
	\min_{k\leq K}\varphi(\x^k)-\min\varphi
\leq
	\frac{C}{\sum_{k=1}^{K+1}\gamk}
\]
for some \(C>0\) confirms, see \cref{thm:sumgamk}, larger stepsizes are typically associated with faster convergence, suggesting that selecting ``\(=\)'' in \eqref{eq:adaPG} would be preferable.

\subsection{Motivating examples: Barzilai-Borwein stepsizes}
	A first motivation for this work is the observation that the above argument may be shortsighted.
	To gain an intuition, consider \(\rp=1\) in \eqref{eq:adaPG} and assume that \(\nabla f\) is globally \(L_f\)-Lipschitz continuous for simplicity; as shown in \cite{latafat2023convergence}, a lower bound \(\gamk\geq\frac{1}{\sqrt{2}L_f}\) can easily be derived.
	This estimate can in fact be tightened by considering the cases in which the stepsize increases or not:
	using \(\ell_k\leq L_k\leq c_k\leq L_f\) and \(\ell_kc_k=L_k^2\) \cite[Lem. 2.1]{latafat2023adaptive} it is easy to see that
	\[
		\gamk*\geq\gamk
	~~\vee~~
			\gamk
		\geq
			\tfrac{1+\sqrt{3}}{2c_k}
	\]
	holds for every \(k\).
	\begin{subequations}\label{eq:adaPGBB}
		By loosening the update into, say,
		\begin{equation}\label{eq:adaPGBBlong}
			\gamk*
		=
			\min\set{
				\gamk\sqrt{1+\tfrac{\gamk}{\gam_{k-1}}},
				\tfrac{\gamk}{
					\sqrt{2\left[\gamk^2L_k^2-\gamk\ell_k\right]_+}
				}{\red{},
				\tfrac{1}{\ell_k}}
			},
		\end{equation}
		a similar comparison yields
		\[
			\gamk*\geq\gamk
		~~\vee~~
			\gamk
		\geq
			\tfrac{3}{2c_k}
		~~\vee~~
			\gamk\geq\gamk*=\tfrac{1}{\ell_k}.
		\]
		Similarly, reducing the last term as
		\begin{equation}\label{eq:adaPGBBshort}
			\gamk*
		=
			\min\set{
				\gamk\sqrt{1+\tfrac{\gamk}{\gam_{k-1}}},
				\tfrac{\gamk}{
					\sqrt{2\left[\gamk^2L_k^2-\gamk\ell_k\right]_+}
				}{\red{},
				\tfrac{1}{c_k}}
			}
		\end{equation}
		benefits the second case, with new lower bounds given by
		\[
			\gamk*\geq\gamk
		~~\vee~~
			\gamk
		\geq
			\tfrac{2}{c_k}
		~~\vee~~
			\gamk\geq\gamk*=\tfrac{1}{c_k}.
		\]
	\end{subequations}

	Interestingly, \eqref{eq:adaPGBB} amount to \emph{dampened} versions of the celebrated \emph{long} and \emph{short Barzilai-Borwein} (BB) stepsizes \(\nicefrac{1}{\ell_k}\) and \(\nicefrac{1}{c_k}\) \cite{barzilai1988two}.
	It is easy to see that through the updates \eqref{eq:PG} and \eqref{eq:adaPGBB} the generated stepsize sequences are bounded away from zero; then, since \(\gamk*\) in both is dominated by the right-hand side of \eqref{eq:adaPG} (with \(\rp=1\)), the analysis of \cite{latafat2023convergence} directly applies, resulting in a simple update rule to make BB stepsizes globally convergent, for possibly nonsmooth convex problems.

	This is the sense in which we say that \adaPG{} acts as a \emph{``safeguard''} for BB stepsizes, namely in providing a suitable dampening to trigger global convergence.
	These claims will be substantiated with a formal proof addressing a richer choice of stepsizes spanning beyond long and short BB.
%


\subsection{The locally H\"older case}
	As a matter of fact, the (locally) Lipschitz differentiable case is a no-brainer consequence of \cite[Thm. 1.1]{latafat2023convergence}.
	Nevertheless, as the numerical evidence demonstrates, such a simple change in the update rule can have a strong impact.
	A deeper theoretical contribution in this paper is provided by the \emph{locally H\"older} analysis which builds upon the recent developments of \cite{oikonomidis2024adaptive}.
	In this setting, the theory of \cite{oikonomidis2024adaptive} cannot be invoked, since the proofs therein are tightly linked to the update rule \eqref{eq:adaPG}, complicated by possibly vanishing stepsizes.
	The proposed solution is a convergence recipe prescribing two conditions: (i) do not overshoot the update \eqref{eq:adaPG} of \adaPG, and (ii) control the stepsize from below based on a past fixed-point residual.
	Surprisingly, not only does this solution prove that \adaPG{} can \emph{safeguard} even in the locally H\"older regime, but it also improves the results in \cite{oikonomidis2024adaptive} for \adaPG{} itself by providing tighter rate constants.



	\section{Convergence of adaptive methods revisited}
		\label{sec:recipe}%
		In what follows we
adopt the same notation of \cite{latafat2023adaptive}:
\begin{equation}
	\rhok\coloneqq\tfrac{\gamk}{\gam_{k-1}},
~~
	P_k\coloneqq\varphi(\xk)-\min\varphi,
~~
	P_k^{\rm min}\coloneqq\min_{i\leq k}P_i.
\end{equation}
Moreover, following \cite{oikonomidis2024adaptive}, for \(\q\in[0,1]\) we introduce
\begin{equation}\label{eq:lamk}
	\lam_k
\coloneqq
	\frac{\gamk}{\norm{\xk-\x^{k-1}}^{1-\q}}
\end{equation}
acting as a \emph{scaled} stepsize, and the local H\"older estimates
\begin{equation}\label{eq:lL}
	\lk
\coloneqq
	\ell_k\norm{\xk-\x^{k-1}}^{1-\q}
~~\text{and}~~
	\Lk
\coloneqq
	L_k\norm{\xk-\x^{k-1}}^{1-\q}.
\end{equation}
When \(\nabla f\) is locally \(\q\)-H\"older continuous, for any compact convex set \(\Omega\subset\R^n\) there exists a \(\q\)-H\"older modulus \(\L_{\Omega}>0\) for \(\nabla f\) on \(\Omega\), namely such that
\[
	\norm{\nabla f(\x)-\nabla f(y)}
\leq
	\L_{\Omega}
	\norm{\x-y}^{\q}
\quad
	\forall\x,y\in\Omega.
\]
Then, as long as \(\seq{\xk}\) remains in a bounded and convex set \(\Omega\), one has that \cite[Eq. (13)]{oikonomidis2024adaptive}
\begin{equation}\label{eq:lLL}
	\l_k\leq\L_k\leq\L_{\Omega}.
\end{equation}

\subsection{A convergence recipe}
	We now list the fundamental ingredients of the convergence analysis.
	After stating some preliminary lemmas, we will prove that the given items are all is needed for ensuring convergence of proximal gradient iterations in the generality of \cref{ass:basic}.
	In the following \cref{sec:stepsizes} we will translate these requirements into conditions on stepsize oracles that can be \emph{safeguarded}, followed by some notable examples.

	\begin{center}
		\fbox{\parbox{.97\linewidth}{%
			\textbf{Convergence recipe}

			Let a sequence \(\seq{\xk}\) be generated by proximal gradient iterations \eqref{eq:PG} with stepsizes \(\seq{\gamk}\subset\R_{++}\), and let \(\ell_k\) and \(L_k\) be as in \eqref{eq:lL_noq}.
			We say that \(\seq{\xk}\) and \(\seq{\gamk}\) satisfy \cref{prop:adaPG1} (resp. \ref{prop:adaPG2}/\ref{prop:lammin}) if there exist \(\q\in(0,1]\), \(\rp\in[1,2]\) and \(\lam_{\rm min}>0\) such that, for all \(k\in\N\):
			\begin{enumeratprop}[topsep=5pt]
			\item \label{prop:adaPG1}%
				\(1+\rp \rhok- \rp\rhok*^2\geq0\)
			\item \label{prop:adaPG2}%
				\(
					\tfrac{1}{2}
					-
					\rhok*^2\bigl[
						\gamk^2L_k^2
						-
						\gamk\ell_k(2-\rp)
						+
						1-\rp
					\bigr]
				\geq
					0
				\)
			\item \label{prop:lammin}%
				either \(\gamk*\geq\gamk\) or there exists and index \(j_k\leq k\) such that \(\min\set{\gam_{j_k},\gamk*}\geq\lam_{\rm min}\norm{x^{j_k}-x^{j_k-1}}^{1-\q}\).
			\end{enumeratprop}
		}}%
	\end{center}

	\Cref{prop:adaPG1,prop:adaPG2} dictate that the stepsize should not overshoot the one in \eqref{eq:adaPG}: \adaPG{} shall act as \emph{safeguard}.
	As stated in \cref{thm:P1P2}, these two alone already ensure boundedness of the generated iterates \(\xk\).
	The remaining \cref{prop:lammin} demands a bound from below: for \(\q=1\), it is tantamount to requiring \(\seq{\gamk}\) bounded away from zero;
	more generally, it involves a bound on a \emph{scaled} stepsize \(\lam_i\), cf. \eqref{eq:lamk}, akin to that observed in \cite[Lem. 3.6]{oikonomidis2024adaptive}; see \cref{thm:properties}.
	The turning point lies in allowing full flexibility for the index \(j_k\leq k\) in \cref{prop:lammin}, enabling us to safeguard stepsizes that depend on a window of past iterations, as will be showcased in \cref{sec:stepsizes}.

\subsection{Preliminary results on adaptive methods}
	\begin{fact}[{\cite[Lem. 3.3]{oikonomidis2024adaptive}}]%
		Let \(f\) and \(g\) be convex, \(\rp>0\) and \(\x^\star\in\argmin\varphi\).
		Relative to \(\seq{\xk}\) generated by proximal gradient iterations \eqref{eq:PG} with stepsizes \(\gamk>0\) \(\forall k\in\N\), let
		\begin{equation}
			\Uk(\x^\star)
		\coloneqq
			\tfrac{1}{2}\norm{\xk-\x^\star}^2
			+
			\tfrac{1}{2}\norm{\xk-\x^{k-1}}^2
			+
			\gamk(1+\rp\rhok)P_{k-1}.
		\label{eq:Uk}
		\end{equation}
		Then, for every \(k\geq1\) and with \(\ell_k\) and \(L_k\) as in \eqref{eq:lL_noq},
		\begin{align*}
			\Uk*(\x^\star)
		\leq{} &
			\Uk(\x^\star)
			-
			\gamk\bigl(1+\rp\rhok-\rp\rhok*^2\bigr)P_{k-1}
		\\
		&
			-
			\Bigr\{
				\tfrac12
				-
				\rhok*^2\bigl[
					\gamk^2L_k^2
					-
					\gamk\ell_k(2-\rp)
		\\
		&
			\hspace*{2.0cm}
					+
					1-\rp
				\bigr]
			\Bigr\}
			\norm{\xk-\x^{k-1}}^2.
		\numberthis\label{eq:SD}
		\end{align*}
	\end{fact}

	\begin{fact}[{\cite[Lem. 3.4]{oikonomidis2024adaptive}}]\label{thm:P1P2}%
		Let \cref{ass:basic} hold with \(\q\in[0,1]\), \(\x^\star\in\argmin\varphi\), and \(\seq{\xk}\) be generated by proximal gradient iterations \eqref{eq:PG} with stepsizes \(\gamk>0\) \(\forall k\in\N\).
		If \(\seq{\xk}\) and \(\seq{\gamk}\) satisfy \cref{prop:adaPG1,prop:adaPG2}, then:
		\begin{enumerate}
		\item \label{thm:descent}%
			\(\seq{\U_k(\x^\star)}\) decreases and converges to a finite value.

		\item \label{thm:bounded}%
			The sequence \(\seq{\xk}\) is bounded and admits at most one optimal limit point.

		\item \label{thm:sumgamk}%
			\(P_K^{\rm min}\leq\U_1(\x^\star)\big/\bigl(\sum_{k=1}^{K+1}\gamk\bigr)\) for any \(K\geq1\).
		\end{enumerate}
	\end{fact}
	\begin{proof}
		The proof is identical to that of \cite[Lem. 3.4]{oikonomidis2024adaptive}, after observing that \cref{prop:adaPG1,prop:adaPG2} together amount to
		\[
			\gamk*
		\leq
			\gamk\min\set{
				\sqrt{\tfrac{1}{\rp}+\tfrac{\gamk}{\gam_{k-1}}},
				\tfrac{1}{
					\sqrt{2\left[\gamk^2L_k^2-(2-\rp)\gamk\ell_k+1-\rp\right]_+}
				}
			}.
		\]
	\end{proof}

	\begin{lemma}[compliance of \adaPG]\label{thm:properties}%
		Let \cref{ass:basic} hold with \(\q\in[0,1]\).
		Then, \(\seq{\xk}\) and \(\seq{\gamk}\) generated by \adaPG\ (\eqref{eq:PG} and \eqref{eq:adaPG}) satisfy \cref{prop:adaPG1,prop:adaPG2,prop:lammin} with
		\[
			j_k\in\set{k-1,k}
		\quad\text{and}\quad
			\lam_{\rm min}=\tfrac{1}{\sqrt{2\rp}\L_\Omega},
		\]
		where \(\L_\Omega\) is a \(\q\)-H\"older modulus for \(\nabla f\) on a compact and convex set \(\Omega\) that contains all the iterates \(\xk\).
	\end{lemma}
	\begin{proof}
		Compliance with \cref{prop:adaPG1,prop:adaPG2}, as well as the existence of such an \(\Omega\) and the consequent finiteness of \(\L_\Omega\), follows from \cref{thm:P1P2}.
		Recall in particular that \(\lk\leq\Lk\leq\L_\Omega\) holds for every \(k\), cf. \eqref{eq:lLL}.
		Suppose that \(\gamk*<\gamk\) at some iteration \(k\), and let \(K_i\subseteq\N\) denote the set of iterates \(k\) for which the \(i\)-th element in the minimum of \eqref{eq:adaPG} is active, \(i=1,2\).
		We have two cases:
		\begin{proofitemize}[topsep=3pt]
		\item
			If \(k\in K_2\), then
			\begin{align*}
				\gamk
			>{} &
				\gamk*
			=
				\tfrac{\gamk}{
					\sqrt{
						2\left[\gamk^2L_k^2-(2-\rp)\gamk\ell_k - (\rp-1)\right]
					}
				}
			\\
			\geq{} &
				\tfrac{1}{\sqrt{2}L_k}
			=
				\tfrac{\norm{\xk-\x^{k-1}}^{1-\q}}{\sqrt{2}\L_k}
			\geq
				\tfrac{\norm{\xk-\x^{k-1}}^{1-\q}}{\sqrt{2}\L_\Omega}.
			\numberthis\label{eq:lamminK2}
			\end{align*}
			Noticing that \(\frac{1}{\sqrt{2}\L_\Omega}\geq\lam_{\rm min}\) (since \(\rp\geq1\)), apparently the claimed \cref{prop:lammin} holds with \(j_k=k\).

		\item
			If \(k\in K_1\) (necessarily \(\rp>1\)), then \(1>\rhok*^2=\frac{1}{\rp}+\rhok\), which yields that \(\rhok<1-\frac{1}{\rp}<1\) and that consequently \(\gamk*<\gamk<\gam_{k-1}\).
			Necessarily \(k-1\in K_2\), for otherwise, denoting \(t=\rp^{-1}\geq\frac12\), it would hold that \(\rhok*=\sqrt{t+\rhok}=\sqrt{t+\sqrt{t+\rho_{k-1}}}\geq\sqrt{\nicefrac{1}{2}+\nicefrac{1}{\sqrt{2}}}>1\).
			Thus,
			\[
				\gam_{k-1}
			>
				\gamk*
			=
				\gamk
				\sqrt{\tfrac{1}{\rp}+\rhok}
			\geq
				\tfrac{1}{\sqrt{\rp}}
				\gamk
			\overrel[>]{\eqref{eq:lamminK2}}
				\tfrac{\norm{\x^{k-1}-x^{k-2}}^{1-\q}}{\sqrt{2\rp}\L_\Omega}.
			\]
			In this case, the same conclusion holds with \(j_k=k-1\).
		\qedhere
		\end{proofitemize}
	\end{proof}

%

\subsection{Convergence analysis}
	We now present the main result of this paper, which generalizes and improves upon the analysis of \cite{oikonomidis2024adaptive}.
	To minimize overlapping arguments, proofs of intermediate claims that are verbatim identical are deferred to the reference.

	\begin{theorem}\label{thm:convergence}%
		Let \cref{ass:basic} hold (with \(\q>0\)), and let \(\seq{\xk}\) be generated by proximal gradient iterations \eqref{eq:PG}.
		If all \cref{prop:adaPG1,prop:adaPG2,prop:lammin} hold for some \(\rp\in[1,2]\) and \(\lam_{\rm min}>0\), then \(\seq{\xk}\) converges to some \(\x^\star \in \argmin \varphi\) with
		\[
			P_K^{\rm min}
		\leq
			\max\set{
				\tfrac{\U_1(\x^\star)}{\gam_0(K+1)},
				\tfrac{
					2^{\frac{1-\q}{2}}
					\U_1(\x^\star)^{\frac{1+\q}{2}}
					(1+\lam_{\rm min}\L_{\Omega})^{1-\q}
				}{
					\lam_{\rm min}
					(K+1)^{\q}
				}
			}
		\]
		for every \(K\geq1\), where \(\U(\x^\star)\) is as in \eqref{eq:Uk} and \(\L_{\Omega}\) is a \(\q\)-H\"older modulus for \(\nabla f\) on a convex and compact set \(\Omega\) that contains all the iterates \(\xk\).
	\end{theorem}
	\begin{proof}
		The assumptions of \cref{thm:P1P2} are met, and therefore all the claims therein hold.
		In particular, the existence of a \(\q\)-H\"older modulus \(\L_{\Omega}\) for \(\nabla f\) on a convex and compact set \(\Omega\) that contains all the iterates \(\xk\) is guaranteed.
		\Cref{prop:adaPG1,prop:adaPG2} ensure that the coefficients of \(P_{k-1}\) and \(\norm{\x^k-\x^{k-1}}^2\) in \eqref{eq:SD} are negative.
		A telescoping argument then yields that
		\begin{equation}\label{eq:sumPk}
			\sum_{k\geq1}\gamk\bigl(1+\rp\rhok-\rp\rhok*^2\bigr)P_{k-1}
		<
			\infty.
		\end{equation}
		We next proceed by intermediate steps.
		In what follows, we denote \(\K\coloneqq\set{k\in\N}[\gamk*\geq\gamk]\) and \(\K*\coloneqq\N\setminus \K\).
		\def\currentlabel{thm:convergence}%
		\begin{claims}
		\item \label{thm:Pkto0}%
			{\em \(\inf_{k\in\N}P_k=0\), and \(\seq{\xk}\) admits a (unique) optimal limit point.}

			Boundedness of \(\seq{\xk}\) and at most uniqueness of the optimal limit point follow from \cref{thm:bounded}.
			By lower semicontinuity of \(\varphi\), in order to show existence it suffices to prove that \(\inf_{k\in\N}P_k=0\).
			If \(\gamk\not\to0\), then \(\sum_k\gamk=\infty\) and we know from \cref{thm:sumgamk} that \(P_k\to0\).
			Suppose instead that \(\gamk\to0\), so that in particular \(\K*\) must be infinite.
			If \(\K\) is finite, then \(1+\rp\rhok-\rp\rhok*^2\not\to 0\), for otherwise \(\seq{\gamk}\) would be eventually linearly increasing.
			If \(\K\) is infinite, then so is \(\tilde K\coloneqq\set{k\in \K*}[k-1\in \K]\) with \(\rhok\geq1\) and \(\rhok*<1\) for all \(k\in\tilde K\);
			in particular,
			\(
				1+\rp\rhok-\rp\rhok*^2
			\geq
				1
			\)
			for all \(k\in\tilde K\).
			Also in this case we conclude that \(0\leq1+\rp\rhok-\rp\rhok*^2\not\to 0\) as \(\K*\ni k\to\infty\), where the inequality is ensured by \cref{prop:adaPG1}.
			Then, (regardless of whether \(\K\) is finite or not) there exists \(\varepsilon>0\) together with an infinite set \(\hat K\subseteq \K*\) such that \(1+\rp\rhok-\rp\rhok*^2\geq\varepsilon\) holds for all \(k\in\hat K\).
			Since \(\hat K\subseteq\K*\), by virtue of \cref{prop:lammin} we then have that
			\(
				1+\rp\rhok-\rp\rhok*^2\geq\varepsilon
			\)
			and
			\(
				\gamk*
			\geq
				\lam_{\rm min}\norm{x^{j_k}-x^{j_k-1}}^{1-\q}
			\)
			hold
			\(\forall k\in\hat K\),
			where \(j_k\leq k\) for every \(k\in\hat K\).
			By combining with \eqref{eq:sumPk} we obtain
			\[
				\infty
			>
				\sum_{k\in\hat K}\gamk* P_k
			\geq
				\lam_{\rm min}\sum_{k\in\hat K}\norm{x^{j_k}-x^{j_k-1}}^{1-\q}P_{k}.
			\]
			Therefore, either \(\liminf_{\hat K\ni k\to\infty}\norm{x^{j_k}-x^{j_k-1}}^{1-\q}=0\) or \(\liminf_{\hat K\ni k\to\infty}P_{k}=0\).
			In the latter case, the claim is proven.
			In the former case, necessarily (\(\q<1\) and) the sequence \(\seq{j_k}[k\in\hat K]\) contains infinitely many indices (for otherwise \(\seq{\norm{x^{j_k}-x^{j_k-1}}}[k\in\hat K]\) would alternate between a finite number of nonzero elements).
			Up to extracting, we have that \(\lim_{\hat K\ni k\to\infty}\norm{x^{j_k}-x^{j_k-1}}=0\).
			For any \(\x^\star\in\argmin\varphi\) we thus have (recall that \(P_k=\varphi(\x^k)-\varphi(\x^\star)\))
			\begin{align*}
				P_{j_k}
			\leq{} &
				\innprod{x^{j_k}-\x^\star}{\tfrac{x^{j_k-1}-x^{j_k}}{\gam_{j_k}}-\bigl(\nabla f(x^{j_k-1})-\nabla f(x^{j_k})\bigr)}
			\\
			\leq{} &
				\norm{x^{j_k}-\x^\star}
				\Bigl(
					\tfrac{1}{\gam_{j_k}}
					\norm{x^{j_k-1}-x^{j_k}}
				\\
			&
			\numberthis\label{eq:Pkbound}
				\hphantom{
					\norm{x^{j_k}-\x^\star}
					\Bigl(
				}
					+
					\norm{\nabla f(x^{j_k-1})-\nabla f(x^{j_k})}
				\Bigr)
			\\
			\leq{} &
				\norm{x^{j_k}-\x^\star}
				\norm{x^{j_k-1}-x^{j_k}}^{\q}
				\tfrac{1+\lam_{\rm min}\L_\Omega}{\lam_{\rm min}}
			\quad
				\forall k\in\hat K.
			\end{align*}
			Since \(\q>0\), by taking the limit as \(\hat K\ni k\to\infty\) we obtain that \(\lim_{\hat K\ni k\to\infty}P_{j_k}=0\).

	%
	%
	%
		\item
			{\em \(\seq{\xk}\) converges to a solution.}

			Having shown the previous claim and since \(\sup_{k\in\N}\rhok<\infty\) (by \cref{prop:adaPG1}), the proof is the same as in \cite[Thm. 3.8]{oikonomidis2024adaptive}.
		\end{claims}
		To simplify the notation we let \(D\coloneqq\sqrt{2\U_1(\x^\star)}\).
		\begin{claims}[resume]
	%

		\item \label{thm:sublinear:Pk<Dk}%
			{\em
				\(
					P_k^{\rm min}
				\leq
					D\tfrac{1+\lamk\L_\Omega}{\lamk}
					\norm{\xk-\x^{k-1}}^{\q}
				\)
				holds for any \(k\in\N\).
			}

			See \cite[Claim 3.9(b)]{oikonomidis2024adaptive}.

		\item
			{\em
				For any \(k\in\N\) it holds that
				\[
					\gamk*
				\geq
					\begin{ifcases}
						\lam_{\rm min}^{\frac{1}{\q}}
						\Bigl(
							\tfrac{P_k^{\rm min}}{D(1+\lam_{\rm min}\L_{\Omega})}
						\Bigr)^{\hspace{-0.25em}\frac{1-\q}{\q}}
					&
						k \geq\min\K*
					\\[3pt]
						\gam_0
					\otherwise.
					\end{ifcases}
				\]
			}

			Suppose first that \(k\in\K*\) (in particular, \(k\geq\min\K*\)).
			Since \(\gamk*<\gamk\), it follows from \cref{prop:lammin} that
			\[
				\gamk*\geq\min\set{\gam_{j_k},\gamk*}
			\geq
				\lam_{\rm min}\norm{x^{j_k}-x^{j_k-1}}^{1-\q}
			\]
			holds for some \(j_k\leq k\), and in particular that \(\lam_{j_k}\geq\lam_{\rm min}\).
			Noticing that \(P_{j_k}^{\rm min}\geq P_k^{\rm min}\) since \(j_k\leq k\), the sought inequality follows by using the lower bound
			\[
				\norm{x^{j_k}-x^{j_k-1}}
			\geq
				\Bigl(\tfrac{\lam_{j_k}P_{j_k}^{\rm min}}{D(1+\lam_{j_k}\L_\Omega)}\Bigr)
				^{\!\!\frac{1}{\q}}
			\geq
				\Bigl(\tfrac{\lam_{\rm min}P_k^{\rm min}}{D(1+\lam_{\rm min}\L_\Omega)}\Bigr)
				^{\!\!\frac{1}{\q}}
			\]
			obtained from \cref{thm:sublinear:Pk<Dk} raised to the power
			\(\frac{1}{\q}\).

			Next, suppose that \(k\in\K\).
			Let
			\[
				\K*_{<k}
			\coloneqq
				\K*\cap\set{0,1,\dots,k-1}
			\]
			denote the (possibly empty) set of all iteration indices up to \(k-1\) in which the next stepsize is strictly smaller.

			If \(\K*_{<k} = \emptyset\), then \(\gamk\geq\gam_{k-1}\geq\dots\geq\gam_0\).

			Suppose instead that \(\K*_{<k}\neq\emptyset\), and let \(i_k\) denote its largest element:
			\[
				i_k
			\coloneqq
				\max \K*_{<k}
			=
				\max\set{i<k}[
					\gam_{i+1}
				<
					\gam_i
				]
			<
				k.
			\]
			Then,
			\begin{align*}
				\gamk*
			\geq
				\gam_{i_k+1}
			\geq{} &
				\lam_{\rm min}^{\frac{1}{\q}}
				\left(
					\tfrac{
						1
					}{
						D(1+\lam_{\rm min}\L_\Omega)
					}
					P_{i_k}^{\rm min}
				\right)^{\frac{1-\q}{\q}}
			\\
			\geq{} &
				\lam_{\rm min}^{\frac{1}{\q}}
				\left(
					\tfrac{
						1
					}{
						D(1+\lam_{\rm min}\L_\Omega)
					}
					P_k^{\rm min}
				\right)^{\frac{1-\q}{\q}}
			\end{align*}
			where the second inequality follows from the fact that \(i_k\in\K*\), and the last one from the fact that \(\seq{P_k^{\rm min}}\) is decreasing (note that \(k>i_k\)).
		\end{claims}

		The sum of stepsizes can then be lower bounded by
		\begin{align*}
			\sum_{k=1}^{K+1} \gamk
		\geq{} &
			\sum_{k=0}^K
			\min\set{
				\gam_0,
				\lam_{\rm min}^{\frac{1}{\q}}
				\Bigl(
					\tfrac{P_k^{\rm min}}{D(1+\lam_{\rm min}\L_{\Omega})}
				\Bigr)^{\hspace{-0.25em}\frac{1-\q}{\q}}
			}
		\\
		\geq{} &
			(K+1)\min\set{
				\gam_0,
				\lam_{\rm min}^{\frac{1}{\q}}
				\Bigl(
					\tfrac{P_K^{\rm min}}{D(1+\lam_{\rm min}\L_{\Omega})}
				\Bigr)^{\hspace{-0.25em}\frac{1-\q}{\q}}
			},
		\end{align*}
		where the last inequality uses the fact that \(P_K^{\rm min}\leq P_k^{\rm min}\) for every \(k\leq K\).
		Therefore, in light of \cref{thm:sumgamk} we have
		\[
			\U_1(\x^\star)
		\geq
			(K+1)\min\set{
				\gam_0P_K^{\rm min},~
				\tfrac{
					(\lam_{\rm min}P_K^{\rm min})^{\nicefrac{1}{\q}}
				}{
					\left(D(1+\lam_{\rm min}\L_{\Omega})\right)^{\frac{1-\q}{\q}}
				}
			}.
		\]
		From \(D\coloneqq\sqrt{2\U_1(\x^\star)}\), the claimed bound follows.
	\end{proof}

	Neglecting the first term for simplicity, from \cref{thm:properties} we obtain the sublinear rate
	\begin{equation}\label{eq:newrate}
		P_K^{\rm min}
	\leq
		(1+\sqrt{2\rp})^{1-\q}
		\frac{
			\sqrt{2}\sqrt{\rp}^{\q}
			\U_1(\x^\star)^{\frac{1+\q}{2}}
			\L_{\Omega}
		}{(K+1)^\q}
	\end{equation}
	for plain \adaPG, improving the estimate
	\[
		P_K^{\rm min}
	\leq
		(1+\sqrt{2}\rho_{\rm max})^{1-\q}
		\frac{
			\sqrt{2}\sqrt{\rp}^{\q}
			\U_1(\x^\star)^{\frac{1+\q}{2}}
			\L_{\Omega}
		}{(K+1)^\q}
	\]
	with \(\rho_{\rm max}\coloneqq\frac{1+\sqrt{1+\frac{4}{\rp}}}{2}\) of \cite[Thm. 3.9]{oikonomidis2024adaptive}.

	\section{Choice of stepsizes}
		\label{sec:stepsizes}%
		In this section we translate the convergence recipe into sufficient conditions for ``fast'' stepsize choices \(\gamk*^{\fast}\) to be safeguardable by \adaPG.
To this end, we restrict our scrutiny to stepsizes oracles of the form
\[
	\gamk*^{\fast}=\Gamma^{\fast}(\x^{k-m},\dots,\x^k)
\]
for some \(m\geq1\) and \(\func{\Gamma^{\fast}}{(\R^n)^{m+1}}{\R_{++}}\).
The safeguard framework is elementary: take the minimum between the desired stepsize \(\gamk*^{\fast}\) and the \emph{safe} one of \adaPG:

\begin{algorithm}
	\caption{%
		Safeguard for stepsizes \(\mathtight\func{\Gamma^{\fast}}{(\R^n)^{m+1}}{\R_{++}}\)%
	}%
	\label{alg:safe}%
	\begin{algorithmic}[1]
\item[\hspace{7pt}{\sc Require}]
	\begin{tabular}[t]{@{}l@{}}
		\(\rp\in[1,2]\),~
		\(\x^0,\dots,\x^m\in\R^n\),~
		\(\gam_{m-1},\gam_m>0\)
	\end{tabular}
\item[\hspace{7pt}{\sc Repeat for }\(k=m,m+1,\dots\) until convergence]
	\State
		\!\!
		\(
			\gamk*^{\safe}
		=
			\min\set{\!\!
				\gamk\sqrt{\tfrac{1}{\rp}+\tfrac{\gamk}{\gam_{k-1}}},
				\tfrac{\gamk}{
					\sqrt{2\left[\gamk^2L_k^2-(2-\rp)\gamk\ell_k+1-\rp\right]_+}
				}
			\!\!}
		\)
	\State
		\!\!
		\(
			\gamk*^{\fast}
		=
			\Gamma^{\fast}(\x^{k-m},\dots,\xk)
		\)
	\State
		\!\!
		\(
			\gamk*
		=
			\min\set{\gamk*^{\safe},\gamk*^{\fast}}
		\)
	\State
		\!\!
		\(
			\xk*=\prox_{\gamk*g}(\xk-\gamk*\nabla f(\xk))
		\)
\end{algorithmic}

\end{algorithm}


\begin{lemma}[convergence of \cref{alg:safe}]\label{thm:safe}%
	Let \cref{ass:basic} hold for some \(\q\in(0,1]\), and let \(\func{\Gamma^{\fast}}{(\R^n)^{m+1}}{\R_{++}}\) for some \(m\geq1\).
	Suppose that for any compact \(\Omega\subset\R^n\) there exists \(\lam_{\Omega}>0\) with the property that whenever \(z^0,\dots,z^m\in\Omega\) there exists \(i\in\set{1,\dots,m}\) such that \(\Gamma^{\fast}(z^0,\dots,z^m)\geq\lam_{\Omega}\norm{z^i-z^{i-1}}^{1-\q}\).
	Then, denoting \(\rho_{\rm max}=\frac{1+\sqrt{1+\frac{4}{\rp}}}{2}\), the iterates generated by \cref{alg:safe} comply with the convergence recipe with \(\lam_{\rm min}=\min\set{\frac{1}{\sqrt{2\rp}\L_\Omega},\tfrac{\lam_{\Omega}}{\rho_{\rm max}^{m-1}}}\), and in particular \cref{thm:convergence} holds true.
\end{lemma}
\begin{proof}
	\Cref{prop:adaPG1,prop:adaPG2} follow from \(\gamk*\leq\gamk*^{\safe}\).
	\Cref{thm:bounded} then ensures boundedness of \(\seq{\xk}\), whence the existence of a \(\lam_{\Omega}>0\) as in the statement.
	Thus, for every \(k\) there exists \(j_k\in\set{k-m+1,\dots,k}\) such that \(\gamk*^{\fast}\geq\lam_{\Omega}\norm{\x^{j_k}-\x^{j_{k-1}}}^{1-\q}\).
	If \(\gamk*<\gamk*^{\fast}\), \cref{prop:lammin} with \(\lam_{\rm min}=\frac{1}{\sqrt{2\rp}\L_\Omega}\) follows from \cref{thm:properties}.
	Otherwise,
	\[
		\gam_{j_k}
	\geq
		\tfrac{\gamk}{\rho_{\rm max}^{j_k-1}}
	\geq
		\tfrac{\gamk}{\rho_{\rm max}^{m-1}}
	\geq
		\tfrac{\gamk*}{\rho_{\rm max}^{m-1}}
	\geq
		\tfrac{\lam_{\Omega}}{\rho_{\rm max}^{m-1}}
		\norm{\x^{j_k}-\x^{j_{k-1}}}^{1-\q},
	\]
	with first inequality owing to \cref{prop:adaPG1}.
\end{proof}

We now provide some examples of stepsizes that well fit in the safeguarding framework of \cref{alg:safe}.
As will be evident from the experiment in \cref{sec:simulations}, we strongly advocate for the choice of the Anderson acceleration--like one described in \cref{sec:AA}, which dramatically boosts convergence speed in practice.
To simplify the notation, in what follows we denote \(s^k\coloneqq\x^k-\x^{k-1}\) and \(y^k\coloneqq\nabla f(\x^k)-\nabla f(\x)\).

\subsection{Long and short Barzilai-Borwein}\label{sec:step:BB}%
	When \(f\) satisfies \cref{ass:f} for some \(\q\in(0,1]\), these choices respectively correspond to
	\begin{equation}\label{eq:BB}
		\gamk*^{\BB}=\tfrac{1}{\ell_k}
	\quad\text{and}\quad
		\gamk*^{\BB*}=\tfrac{1}{c_k^\q L_k^{1-\q}}=\tfrac{1}{\sqrt{c_k^{1+\q}\ell_k^{1-\q}}}.
	\end{equation}
	It is immediate to verify that the long one complies with \cref{thm:safe}, having \(\ell_k=\lk\norm{\xk-\x^{k-1}}^{1-\q}\leq\L_{\Omega}\norm{\xk-\x^{k-1}}^{1-\q}\), see \eqref{eq:lL} and \eqref{eq:lLL}.
	For the short one, the involvement of a geometric average with the long one when \(\q<1\) is necessary to ensure compliance with \cref{thm:safe}, having
	\(
		\tfrac{1}{c_k}
	=
		\tfrac{\innprod{y^k}{s^k}}{\norm{y^k}^2}
	\propto
		\norm{y^k}^{\frac{1-\q}{\q}};
	\)
	see \cite[Fact 2.2.3]{oikonomidis2024adaptive} for the details.

%

\subsection{Martinez' rule for long and short BB}
	\BB{} and \BB* are one-dimensional equivalents of Broyden's ``good'' and ``bad'' quasi-Newton method, namely
	\[\mathtight
		\gamk*^{\BB}
	=
		\argmin_{\gam\in\R}\norm{\tfrac{s^k}{\gam}-y^k}^2
	~\text{and}~
		\gamk*^{\BB*}
	=
		\argmin_{\gam\in\R}\norm{\gam y^k-s^k}^2
	\]
	are chosen to minimize the secant and inverse secant quasi-Newton approximation \cite{martinez2000practical}.
	As suggested in \cite[\S3.1]{martinez2000practical}, we may choose between the two based on the rule
	\begin{equation}\label{eq:Martinez}
		\gamk*^{\algnamefont Martinez}
	=
		\begin{ifcases}
			\gamk*^{\BB} & \gamk>\frac{\innprod{s^k}{s^{k-1}}}{\innprod{y^k}{y^{k-1}}}
		\\
			\gamk*^{\BB*} \otherwise,
		\end{ifcases}
	\end{equation}
	which selects the one for which the inverse secant error on the previous pair is minimized.

\subsection{Least normalized secant error for long and short BB}
	We can take both direct and inverse errors into account.
	Relative to the previous pair \((s^{k-1},y^{k-1})\), we may opt for \BB{} when it has lower inverse secant error, and \BB* if this has lower secant error.
	If neither holds, we compare the respective scaled residuals, thereby selecting \BB{} when
	\begin{equation}\label{eq:LNSEleq}
		\norm{s^k}^{-1}\norm{s^k-\gamk*^{\BB}y^k}
	\leq
		\norm{y^k}^{-1}\norm{y^k-\nicefrac{s^k}{\gamk*^{\BB*}}}
	\end{equation}
	or else \BB*.
	The proposed rule boils down to
	\begin{equation}\label{eq:LNSE}
		\gamk*^{\algnamefont LNSE}
	=
		\begin{ifcases}
			\gamk*^{\BB} & \gamk*^{\BB}+\gamk*^{\BB*}\leq2\gamk^{\BB*}
		\\
			\gamk*^{\BB*} \otherwise[else if \(\frac{1}{\gamk*^{\BB}}+\frac{1}{\gamk*^{\BB*}}\geq\frac{2}{\gamk^{\BB}}\)]
		\\
			\gamk*^{\BB} \otherwise[else if \eqref{eq:LNSEleq}]
		\\
			\gamk*^{\BB*} \otherwise.
		\end{ifcases}
	\end{equation}

\subsection{Anderson acceleration}\label{sec:AA}%
	Keep the latest \(m\geq1\) pairs \(s^i\) and \(y^i\), and set
	\begin{equation}\label{eq:AA}
		\gamk*^{\AA_m}
	=
		\frac{\sum_{i=k-m+1}^k\innprod{s^i}{y^i}}{\sum_{i=k-m+1}^k\norm{y^i}^2}
	=
		\frac{\sum_{i=k-m+1}^k\norm{y^i}^2\frac{1}{c_i}}{\sum_{i=k-m+1}^k\norm{y^i}^2}.
	\end{equation}
	Apparently, this is a weighted average (with weights \(\norm{y^i}^2\)) of the most recent \BB* stepsizes, hence \cref{thm:safe} is verified when \cref{ass:basic} holds with \(\q=1\).
	When \(\q<1\), consistently with the discussion in \cref{sec:step:BB}, a suitable geometric averaging is needed.
	The name Anderson acceleration \cite{anderson1965iterative} is evocative of its multi-secant (inverse) update interpretation, see \cite{fang2009two} and \cite[\S 3.3.2]{themelis2022douglas}, as it is easily seen to equal
	\(
		\gamk*^{\AA}
	=
		\argmin_{\gam\in\R}\sum_{i=k-m+1}^k\norm{\gam y^i-s^i}^2
	\).

	\section{Numerical results}
		\label{sec:simulations}%
		We conducted a series of simulations to test the effectiveness of \adaPG{} safeguarding the five stepsizes listed in \cref{sec:stepsizes}.
The safeguarding of \adaPG{} is represented in the legend plots with ``{\algnamefont adaPG\(\wedge\)}''.
Our tests are based on the Julia code provided in \cite{latafat2023adaptive} on problems from the LIBSVM dataset \cite{chang2011libsvm}.
In all instances of \adaPG{} we use \(\rp=1.2\), consistently with its good performance reported in \cite{latafat2023convergence}, while for Anderson acceleration \eqref{eq:AA} a memory \(m=4\) is used.
The resulting methods are tested against \adaPG{} \eqref{eq:adaPG} and \adaPBB{} \cite[Alg. 3]{zhou2024adabb}.

We refer to \cite[\S4]{latafat2023adaptive} for a detailed account of the experiments, which are here reproduced with minimal variations.
In each figure, top plots report the best-so-far residual
\[
	r_k
=
	\norm{\tfrac{\xk-\x^{k-1}}{\gamk}-(\nabla f(\x^k)-\nabla f(\x^{k-1}))},
\]
and bottom ones the stepsize cumulative average \(\frac{1}{k}\sum_{i=1}^k\gam_i\).
On the \(x\)-axis we report the number of gradient evaluations, which also corresponds to the iteration count in all methods.

\begin{figure}
	\includetikz[width=\linewidth]{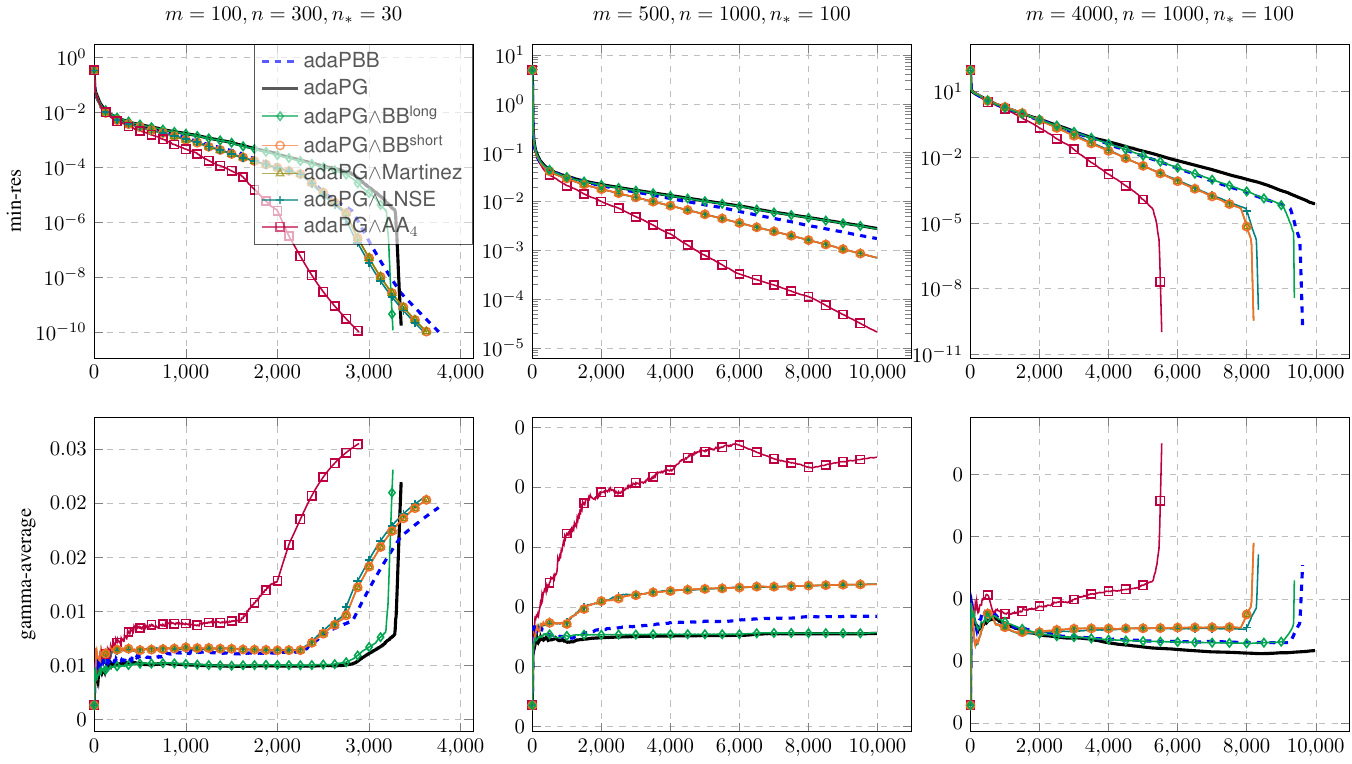}%
	\vspace{-0.75\baselineskip}%
	\caption{%
		Random lasso problem with \(\ell_1\)-regularization parameter \(\lamb=0.1\).
		\(n_\star\) represents the number of nonzero elements in the solution.
	}%
\end{figure}

\begin{figure}
	\includetikz[width=\linewidth]{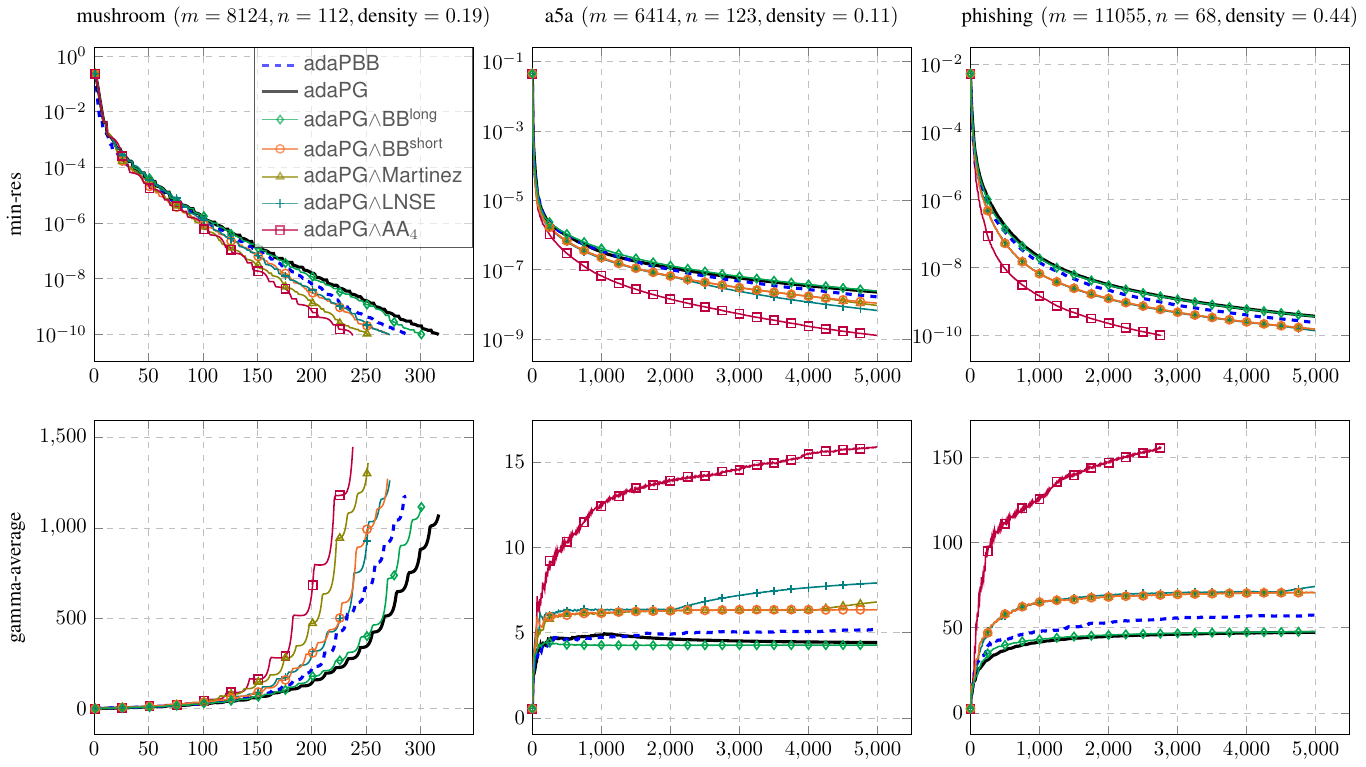}%
	\vspace{-0.75\baselineskip}%
	\caption{%
		Regularized logistic regression (\(m\) and \(n\) are the number of samples and features).
		The \(\ell_1\)-regularization parameter \(\lamb\) is set as in \cite[\S6.1]{zhou2024adabb}.%
	}%
\end{figure}

\begin{figure}
	\includetikz[width=\linewidth]{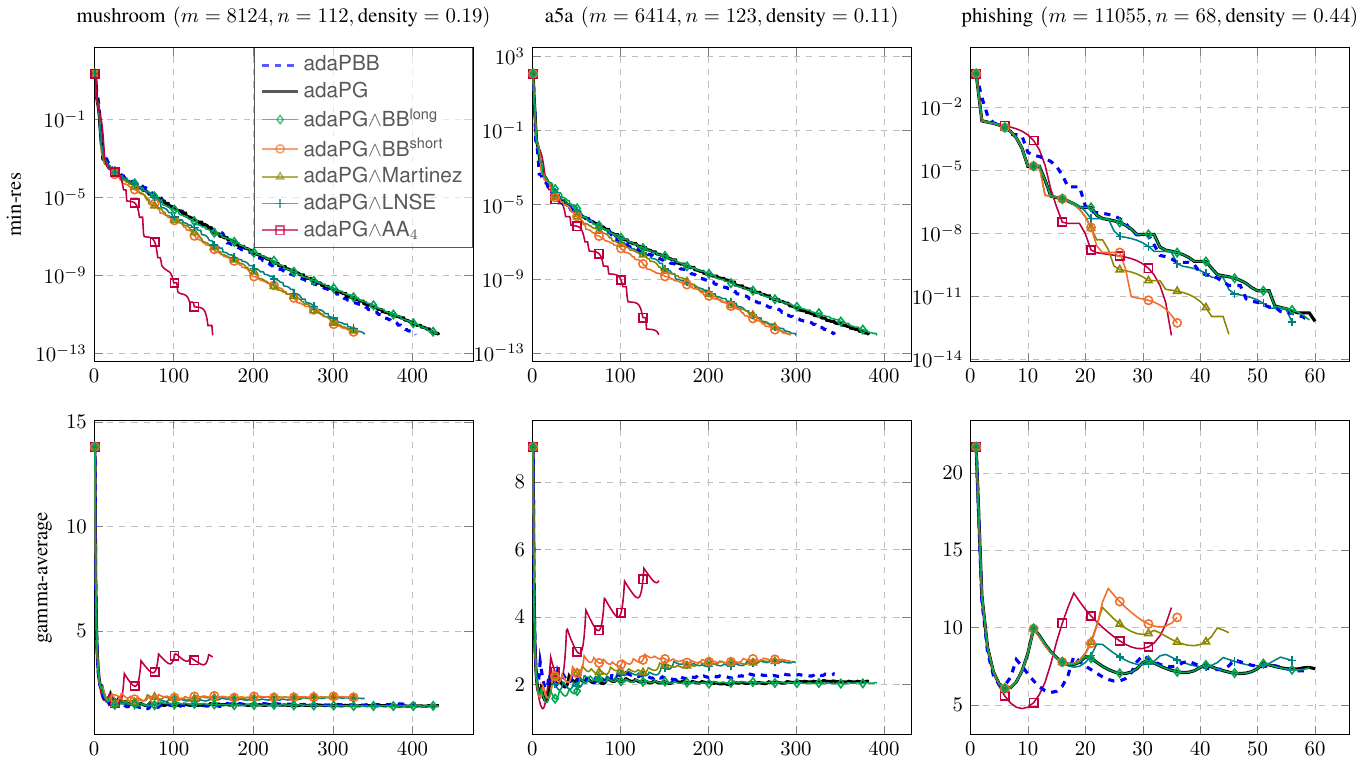}%
	\vspace{-0.75\baselineskip}%
	\caption{%
		Cubic regularization problem with Hessian and gradient generated from the logistic loss problem evaluated at zero on the mushroom and phishing datasets.
		The cubic regularization parameter is set as \(M = 0.01\).
	}%
	\vspace{-.65\baselineskip}%
\end{figure}

All experiments confirm that quasi-Newton--type ideas can yield considerable improvements within adaptive methods.
The convergence speed is evidently correlated with the magnitude of the stepsizes, as confirmed by \cref{thm:sumgamk}.
The winner turns out to be Anderson acceleration \eqref{eq:AA} (\(m=4\)), consistently outperforming all other choices.

	\bibliographystyle{plain}
	\bibliography{TeX/references_abbr.bib}

\begin{thebibliography}{10}

\bibitem{anderson1965iterative}
D.G. Anderson.
\newblock Iterative procedures for nonlinear integral equations.
\newblock {\em J. ACM}, 12(4):547--560, 1965.

\bibitem{barzilai1988two}
J.~Barzilai and J.M. Borwein.
\newblock Two-point step size gradient methods.
\newblock {\em IMA J. Numer. Anal.}, 8(1):141--148, jan 1988.

\bibitem{beck2017first}
A.~Beck.
\newblock {\em First-order methods in optimization}.
\newblock SIAM, 2017.

\bibitem{beck2009fast}
A.~Beck and M.~Teboulle.
\newblock Fast gradient-based algorithms for constrained total variation image
  denoising and deblurring problems.
\newblock {\em IEEE Trans. Image Process.}, 18(11):2419--2434, 2009.

\bibitem{bertsekas2016nonlinear}
D.P. Bertsekas.
\newblock {\em Nonlinear Programming}.
\newblock Athena Scientific, 2016.

\bibitem{boyd2004convex}
S.~Boyd and L.~Vandenberghe.
\newblock {\em Convex Optimization}.
\newblock Cambridge University Press, 2004.

\bibitem{bubeck2014theory}
S.~Bubeck.
\newblock Theory of convex optimization for machine learning.
\newblock {\em \arXivLink{1405.4980}}, 2014.

\bibitem{burdakov2019stabilized}
O.~Burdakov, Y.~Dai, and N.~Huang.
\newblock Stabilized {B}arzilai-{B}orwein method.
\newblock {\em J. Comput. Math.}, 37(6):916--936, 2019.

\bibitem{chang2011libsvm}
C.~Chang and C.~Lin.
\newblock {LIBSVM}: a library for support vector machines.
\newblock {\em ACM Trans. Intell. Syst. Technol. (TIST)}, 2(3):1--27, 2011.

\bibitem{dai2002rlinear}
Y.~Dai and L.~Liao.
\newblock {R}‐linear convergence of the {B}arzilai and {B}orwein gradient
  method.
\newblock {\em IMA J. Numer. Anal.}, 22(1):1--10, 2002.

\bibitem{fang2009two}
H.~Fang and Y.~Saad.
\newblock Two classes of multisecant methods for nonlinear acceleration.
\newblock {\em Numer. Linear Algebra Appl.}, 16(3):197--221, 2009.

\bibitem{latafat2023convergence}
P.~Latafat, A.~Themelis, and P.~Patrinos.
\newblock On the convergence of adaptive first order methods: proximal gradient
  and alternating minimization algorithms.
\newblock {\em \arXivLink{2311.18431}}, 2023.

\bibitem{latafat2023adaptive}
P.~Latafat, A.~Themelis, L.~Stella, and P.~Patrinos.
\newblock Adaptive proximal algorithms for convex optimization under local
  {L}ipschitz continuity of the gradient.
\newblock {\em \arXivLink{2301.04431}}, 2023.

\bibitem{li2021faster}
D.~Li and R.~Sun.
\newblock On a faster {$R$}-linear convergence rate of the {B}arzilai-{B}orwein
  method.
\newblock {\em \arXivLink{2101.00205}}, 2021.

\bibitem{li2022role}
Y.~Li, X.~Tang, X.~Lin, L.~Grzesiak, and X.~Hu.
\newblock The role and application of convex modeling and optimization in
  electrified vehicles.
\newblock {\em Renewable Sustainable Energy Rev.}, 153:111796, 2022.

\bibitem{luo2003applications}
Z.~Luo.
\newblock Applications of convex optimization in signal processing and digital
  communication.
\newblock {\em Math. Program.}, 97(1):177--207, 2003.

\bibitem{malitsky2020adaptive}
Y.~Malitsky and K.~Mishchenko.
\newblock Adaptive gradient descent without descent.
\newblock In {\em Proc. 37th Int. Conf. Mach. Learn.}, volume 119, pages
  6702--6712, 2020.

\bibitem{malitsky2023adaptive}
Y.~Malitsky and K.~Mishchenko.
\newblock Adaptive proximal gradient method for convex optimization.
\newblock {\em \arXivLink{2308.02261}}, 2023.

\bibitem{martinez2000practical}
J.M. Mart{\'i}nez.
\newblock Practical quasi-{N}ewton methods for solving nonlinear systems.
\newblock {\em J. of Comp. Appl. Math.}, 124(1-2):97--121, 2000.

\bibitem{oikonomidis2024adaptive}
K.A. Oikonomidis, E.~Laude, P.~Latafat, A.~Themelis, and P.~Patrinos.
\newblock Adaptive proximal gradient methods are universal without
  approximation.
\newblock {\em \arXivLink{2402.06271}}, 2024.

\bibitem{raydan1993barzilai}
M.~Raydan.
\newblock On the {B}arzilai and {B}orwein choice of steplength for the gradient
  method.
\newblock {\em IMA J. Numer. Anal.}, 13(3):321--326, jul 1993.

\bibitem{themelis2022douglas}
A.~Themelis, L.~Stella, and P.~Patrinos.
\newblock {D}ouglas-{R}achford splitting and {ADMM} for nonconvex optimization:
  Accelerated and {N}ewton-type algorithms.
\newblock {\em Comput. Optim. Appl.}, 82:395--440, 2022.

\bibitem{zhou2024adabb}
D.~Zhou, S.~Ma, and J.~Yang.
\newblock {AdaBB}: Adaptive {B}arzilai-{B}orwein method for convex
  optimization.
\newblock {\em \arXivLink{2401.08024}}, 2024.

\end{thebibliography}
\end{document}